\documentclass[12pt,a4paper]{amsart}
\usepackage{amssymb}
\usepackage{amsfonts}
\usepackage[top=35mm, bottom=35mm, left=30mm, right=30mm]{geometry}
\usepackage[colorlinks=true,citecolor=blue]{hyperref}
\usepackage{mathptmx}
\usepackage{eucal}
\usepackage{graphicx}
\usepackage{array}
\usepackage{mathrsfs}
\usepackage{amssymb}
\usepackage{amsmath}
\usepackage{amsthm}
\usepackage{enumerate}
\usepackage{verbatim}
\usepackage[labelformat=simple]{subcaption}

\usepackage{xcolor}
\newtheorem{thm}{Theorem}[section]

\newtheorem{que}[thm]{Question}
\newtheorem{lem}[thm]{Lemma}
\newtheorem{prop}[thm]{Proposition}

\theoremstyle{definition}

\numberwithin{equation}{section}

\newcommand{\orb}{\mathrm{Orb}}
\newcommand{\Rec}{\mathrm{Rec}}
\newcommand{\Per}{\mathrm{Per}}

\makeatletter
\@namedef{subjclassname@2020}{%
  \textup{2020} Mathematics Subject Classification}
\makeatother

\DeclareMathOperator{\Int}{int}
\DeclareMathOperator{\Bd}{Bd}

\begin{document}
\title{Dendrites and measures with discrete spectrum}
\author[M. Fory\'s]{Magdalena Fory\'s-Krawiec}
\address[M. Fory\'s]{AGH University of Science and Technology, Faculty of Applied
	Mathematics, al. Mickiewicza 30, 30-059 Krak\'ow, Poland}
\email{maforys@agh.edu.pl}
\author[J. Hant\'{a}kov\'{a}]{Jana Hant\'{a}kov\'{a}}
\address[J. Hant\'{a}kov\'{a}]{AGH University of Science and Technology, Faculty of Applied
	Mathematics, al.
	Mickiewicza 30, 30-059 Krak\'ow, Poland
		-- and --
		Mathematical Institute of the Silesian University in Opava, Na Rybn\'i\v{c}ku 1, 74601, Opava, Czech Republic}
\email{jana.hantakova@math.slu.cz}
\author[J. Kupka]{Ji\v{r}\'{\i} Kupka}
\address[J. Kupka]{Centre of Excellence IT4Innovations - Institute for Research and Applications of Fuzzy Modeling, University of Ostrava, 30. dubna 22, 701 03 Ostrava 1, Czech Republic. }
\email{jiri.kupka@osu.cz}
\author[P. Oprocha]{Piotr Oprocha}
\address[P. Oprocha]{AGH University of Science and Technology, Faculty of Applied
	Mathematics, al.
	Mickiewicza 30, 30-059 Krak\'ow, Poland
	-- and --
	Centre of Excellence IT4Innovations - Institute for Research and Applications of Fuzzy Modeling, University of Ostrava, 30. dubna 22, 701 03 Ostrava 1, Czech Republic.}
\email{oprocha@agh.edu.pl}
\author[S. Roth]{Samuel Roth}
\address[S. Roth]{Mathematical Institute, Silesian University in Opava, Na Rybni\v{c}ku 1, 74601, Opava, Czech Republic}
\email{samuel.roth@math.slu.cz}

\begin{abstract}
We are interested in dendrites for which all invariant measures of zero-entropy mappings have discrete spectrum, and we prove that this holds when the closure of the endpoint set of the dendrite is countable. This solves an open question which was around for awhile, almost completing the characterization of dendrites with this property.
\end{abstract}


\maketitle

\section{Introduction}

A \emph{dynamical system} is a pair $(X,f)$ where $X$ is a compact metric space and $f\colon X\to X$ is a continuous map. A \emph{continuum} is a compact connected metric space. Throughout this paper we assume $X$ is a \emph{dendrite}, that is, a locally connected continuum containing no simple closed curve. 

The main motivation of the paper can be derived from the M\"obius Disjointness Conjecture proposed by Sarnak in 2009 \cite{S09}.
By topological arguments the conjecture was confirmed on various one dimensional spaces: the interval \cite{K15}, the circle \cite{Dav}, topological graphs \cite{LOYZ17}, some dendrites \cite{Marzougi}, etc. On the other hand, using ergodic theory it was proved that if all invariant measures have discrete spectrum then the conjecture also holds (e.g. see \cite[Theorem 1.2]{HWY17}). This leads to a natural question, what can be said about the spectrum of measures for zero entropy maps in the above mentioned spaces. In \cite{Opr_Quasi} the authors confirmed that, indeed, topological graphs maps with zero entropy can have only invariant measures with discrete spectrum.
This motivated the following open question \cite[Question 1.1]{Opr_Quasi}: 

\begin{que}\label{que:main}
Which one-dimensional continua $X$ have the property that every invariant measure of $(X,f)$ has discrete spectrum, assuming $f$ is a zero-entropy map?
\end{que}

Similar questions, however, were stated even before, for example in \cite{Scarp} from 1982, where the author asked whether every ergodic invariant measure in a mean equicontinuous system has discrete spectrum. The authors of \cite{Opr_Quasi} have partially answered this question showing the result holds for zero entropy maps on quasi-graphs $X$, and it was completely answered in the affirmative in 2015 in \cite{LTY}.
Let us mention at this point that continua satisfying the condition in Question~\ref{que:main} cannot be too complex. It was shown in \cite{Opr_Quasi} that if a dendrite has an uncountable set of endpoints, then it supports a plethora of maps with zero topological entropy possessing invariant measures which do not have discrete spectrum. Then in the realm of dendrites, only those with a countable set of endpoints can be examples in Question~\ref{que:main}.

In this paper we study the dynamics of zero-entropy dendrite maps on dendrites for which the endpoint set has a countable closure. In Section~\ref{main} we build on results from~\cite{Askri,Askri2} and show that every recurrent point is in fact minimal (Theorem~\ref{thm:recmin}), generalizing a well-known property of zero-entropy interval maps. In Section~\ref{sec:spectrum} we use this result together with a characterization of minimal $\omega$-limit sets from~\cite{Askri2} to show that every invariant measure has discrete spectrum (Theorem~\ref{th:RecMinDiscrete}) in the case of these dendrites. Our results almost completely characterize dendrites for which all invariant measures of zero-entropy mappings have discrete spectrum. We leave unsolved the case of dendrites for which the endpoint set is countable but has an uncountable closure. We strongly believe that in the case of these dendrites the analog of  Theorem~\ref{thm:recmin} also holds, since all known examples seem to confirm that. Unfortunately, we were not able to find a good argument to justify
this statement.
Structural properties of (other) one dimensional continua that may serve as positive examples in Question~\ref{que:main} are yet to be understood. 

\section{Preliminaries}
%\subsection{Topological dynamics}

Let $(X,f)$ be a dynamical system and $x\in X$.
The \textit{orbit of $x$}, denoted by $\orb_f(x)$, is the set $\{f^n(x)\colon n\geq 0\}$, and the \textit{$\omega$-limit set of $x$}, denoted by $\omega_f(x)$, is defined as the intersection $\bigcap_{n\ge 0} \overline{\{f^m (x)\colon  m\ge n\}}$. It is easy to check that $\omega_f(x)$ is closed and strongly $f$-invariant, i.e., $f (\omega_f(x))= \omega_f(x)$. The point $x$ is \textit{periodic} ($x\in \Per(f)$) if $f^p(x) =x$ for some $p\in \mathbb N$, where the smallest such $p$ is called the \textit{period} of $x$. Note that throughout this paper $\mathbb N$ denotes the set of positive integers. The point $x$ is \textit{recurrent} ($x\in \Rec(f)$) if $x\in \omega_f (x)$.
The orbit of a set $A\subset X$, denoted $\orb_f(A)$, is the set $\bigcup_{n\geq 0} f^n(A)$, and $A$ is called \textit{invariant} if $f(A)\subseteq A$.
A set $M$ is \textit{minimal} if it is nonempty, closed, invariant and does not have a proper subset with these three properties. It can be equivalently characterized
by $M=\omega_f(x)$ for every $x\in M$. A point is \textit{minimal} if it belongs to a minimal set.

Recall that \emph{dendrite} is a locally connected continuum $X$ containing no homeomorphic copy of a circle. A continuous map from a dendrite into itself is called a \emph{dendrite map}. For any point $x \in X$ the \emph{order} of $x$, denoted by $ord(x)$, is the number of connected components of $X\setminus \{x\}$. Points of order $1$ are called \emph{endpoints} while points of order at least $3$ are called \emph{branch points}.  By $E(X)$ and $B(X)$ we denote the set of endpoints and branch points respectively. In this paper we especially focus on dendrites in which $E(X)$ has countable closure. These dendrites are a special case of a tame graph, as introduced in~\cite{BBPRV}. Note that when $\overline{E(X)}$ is countable, so also is $\overline{B(X)\cup E(X)}$, since $B(X)$ is countable in any dendrite and has accumulation points only in $\overline{E(X)}$.

For any two distinct points $x,y \in X$ there exists a unique arc $[x,y]\subset X$ joining those points. A \emph{free arc} is an arc containing no branch points.   We say that two arcs $I, J$ form an arc horseshoe for $f$ if $f^n(I) \cap f^m(J)\supset I \cup J$ for some $n,m\in\mathbb{N}$,
where $I, J$ are disjoint except possibly at one endpoint.
Denote by $h_{top}(f)$ the topological entropy of a dendrite map $f$
(for the definition, see \cite{AKM,B,D}). We will frequently use the fact that for dendrite maps positive topological entropy is implied by the existence of an arc horseshoe \cite{KKM}.

The set of all \textit{Borel probability measures} over $X$ is denoted by $M(X)$, and $M_f(X)\subset M(X)$ denotes the set of all elements of $M(X)$  \textit{invariant} with respect to the map $f$.
The set of all \textit{ergodic} measures in $M_f(X)$ is denoted by $M_f^e(X)$. We say that a finite measure $\mu$ on $X$ is \textit{concentrated} on $A\subset X$ if $\mu(A)=\mu(X)$.
It is well known that $M(X)$ endowed with the weak-* topology is a compact metric space and that
$M_f(X)$ is its closed subset.
We say that $\mu\in M_f(X)$ has \textit{discrete spectrum}, if the linear span of the eigenfunctions of $U_f$ in $L^2_\mu(X)$ is dense in $L^2_\mu(X)$, where as usual $U_f$ denotes the \textit{Koopman operator}: $U_f(\varphi)=\varphi\circ f$ for every $\varphi\in L^2_\mu (X)$.
%By a classical result by Ku\v{s}hnirenko, an invariant measure
%has discrete spectrum if and only if it has zero measure-theoretic
%sequence entropy \cite{K67}.
We refer the reader to  \cite{W82} and \cite{Downar} as standard monographs on ergodic theory and entropy.

\section{Recurrence and Minimality in Dendrites with $\overline{E(X)}$ countable} \label{main}
First we recall the following results by Askri on the structure of minimal
$\omega$-limit sets in a special class of dendrite maps.
%\label{thm:decomp}

\begin{prop}[{\cite[Proposition 3.4]{Askri2}}]\label{prop:MinSolSet}
	Let $X$ be a dendrite such that
	$E(X)$ is countable and
	let $f\colon X\to X$ be a continuous map with zero topological entropy. If
	$M= \omega_f (x)$ is an infinite minimal $\omega$-limit set for some $x\in X$, then for every $k\geq 1$ there is an $f$-periodic subdendrite $D_k$ of $X$ and an integer $n_k\geq 2$  with the following
	properties:
	\begin{enumerate}
		\item $D_k$ has period $\alpha_k:=n_1 n_2 \dots n_k$,
		\item for $i\neq j\in\{0,\ldots,\alpha_k-1\}, f^i(D_k)$ and $ f^j(D_k)$ are either disjoint or intersect in one common point, 
		\item $\bigcup_{k=0}^{n_j-1}f^{k\alpha_{j-1}}(D_j)\subset D_{j-1},$	
		\item $M \subset \bigcap_{k\geq 0}Orb_f(D_k)$,
		\item\label{cond:5} $f(M^k_i) = M^k_{i+1 \mod \alpha_k}$, where $M^k_i = M\cap f^i(D_k)$ for all $k$ and all $0\leq i\leq \alpha_k-1$.
	\end{enumerate}
\end{prop}

While \eqref{cond:5} is not directly stated in \cite{Askri2}, it is an obvious consequence of the other statements.

Implicit in Proposition~\ref{prop:MinSolSet} is the idea that the minimal set $M$ has an odometer as a factor. Our next lemma shows that when $E(X)$ has countable closure, the factor map is invertible except on a countable set.

Given an increasing sequence $(\alpha_k)$ with $\alpha_k | \alpha_{k+1}$ for all $k$, we define the group $\Omega=\Omega(\alpha_k)$ of all $\theta\in\prod_{k=0}^\infty\mathbb{Z}/\alpha_k\mathbb{Z}$ such that $\theta_{k+1}$ is congruent to $\theta_k$ modulo $\alpha_k$ for all $k$, and we let $\tau$ denote the group rotation $\tau(\theta)=\theta+(1,1,1,\cdots)$. Then $(\Omega,\tau)$ is called the \emph{odometer} associated to the sequence $(\alpha_k)$.

\begin{lem}\label{lem:almostinvertible}
Let $X,f,M,(D_k),(\alpha_k)$ be as in Proposition~\ref{prop:MinSolSet} and suppose that $\overline{E(X)}$ is countable. Then
\begin{enumerate}
\item The sets $J_\theta=\bigcap_k f^{\theta_k}(D_k)$, $\theta\in\Omega$, are closed, connected, and pairwise disjoint.
%{\color{red} Jana:Does it mean that for the covering dendrites $f^i(D_k)$ and $ f^j(D_k)$ are disjoint or equal, for large enough $k$?}{\color{blue}Piotr:not necessarily}
\item There is a factor map $\pi:(M,f)\to(\Omega,\tau)$ which takes the value $\theta$ on $M\cap J_\theta$.
\item Each fiber $\pi^{-1}(\theta)$, $\theta\in\Omega$, is countable, and all but countably many of these fibers are singletons.
\end{enumerate}
\end{lem}
\begin{proof}
It is clear from Proposition~\ref{prop:MinSolSet} that each set $J_\theta$ is closed, connected, and has non-empty intersection with $M$. It is also clear that $f(J_\theta)=J_{\tau(\theta)}$. However, since the sets $f^i(D_k)\cap f^j(D_k)$ are allowed to intersect at a point, it is not clear if the sets $J_\theta$ are pairwise disjoint. We prove this fact now. Suppose there are $\theta\neq\theta'$ with $J_\theta\cap J_{\theta'}\neq\emptyset$. Find $k$ minimal such that $\theta_k\neq\theta'_k$. Then clearly
\begin{equation*}
J_\theta \cap J_{\theta'} = f^{\theta_k}(D_k) \cap f^{\theta'_k}(D_k) = \{a\}
\end{equation*}
for some single point $a\in X$. Taking the image by $f^{\alpha_k}$ we have
\begin{equation*}
f^{\alpha_k}(a) \in f^{\alpha_k}(J_\theta)\cap f^{\alpha_k}(J_{\theta'}) = J_{\tau^{\alpha_k}(\theta)} \cap J_{\tau^{\alpha_k}(\theta')} = f^{\theta_k+\alpha_k}(D_k) \cap f^{\theta'_k+\alpha_k}(D_k) = \{a\},
\end{equation*}
since $D_k$ is periodic with period $\alpha_k$. This shows that $a$ is periodic with period $\alpha_k$. In particular, it does not belong to the infinite minimal set $M$. Now for $n\in\mathbb{N}$ let $J_n=J_{\tau^{n\alpha_k}(\theta)}$. Then $a\in J_n$ and we can choose an additional point $m_n \in M \cap J_n$ for all $n$. Thus the $J_n$'s are nondegenerate subdendrites and intersect pairwise only at $a$. In particular, the sets $(a,m_n]$ are pairwise disjoint connected subsets of $X$, so their diameters must converge to zero (see eg.~\cite[Lemma 2.3]{MaiShi}). But since $M$ is closed, this shows that $a\in M$, a contradiction.

Now that the sets $J_\theta$, $\theta\in\Omega$ have been shown to be pairwise disjoint, we see immediately that $\pi$ is well-defined. It is also easy to see that $\pi$ is continuous and $\tau\circ\pi = \pi\circ f$.

Again, since the sets $J_\theta$, $\theta\in\Omega$ are pairwise disjoint connected sets in $X$, only countably many of them can have positive diameter. It follows that $\pi^{-1}(\theta)$ is a singleton except for countably many $\theta$. It remains to show that $M\cap J_\theta$ is countable when $J_\theta$ is non-degenerate. Since $M$ is minimal and $J_\theta$ never returns to itself, we must have $M\cap J_\theta \subset \Bd(J_\theta)$, where $\Bd(J_{\theta})$ stands for the boundary of $J_{\theta}$. But the boundary in $X$ of the subdendrite $J_\theta$ is a subset of $E(J_\theta) \cup B(X) \cup \overline{E(X)}$, which is countable by the assumption that $E(X)$ has countable closure. Here we use the well-known facts that $B(X)$ is countable in any dendrite, and the cardinality of the endpoint set of a dendrite cannot increase when we pass to a subdendrite, see e.g.~\cite{Nad, Nagh}.
\end{proof}

The next Lemma strengthens~\cite[Lemma 3.5]{Askri2} by relaxing the condition that $E(X)$ be closed.
\begin{lem}[{\cite[Lemma 3.5]{Askri2}}]\label{lm:decomp}
Let $X,f,M,(D_k),(\alpha_k)$ be as in Proposition~\ref{prop:MinSolSet} and suppose that $\overline{E(X)}$ is countable. Then there is $N\geq 1$ such that $\forall k\geq N, f^{i_k}(D_k)$ is a free arc for some $0\leq i_k\leq \alpha_k-1$.
\end{lem}
\begin{proof}
Using Lemma~\ref{lem:almostinvertible} we know that there are uncountably many singleton sets $J_\theta$. Now a dendrite whose endpoint set has countable closure is always the union of a countable sequence of free arcs and a countable set, see~\cite[Theorem 2.2]{BBPRV}. It follows that we can find $\theta$ with the singleton $J_\theta$ in the interior of a free arc $A$ in $X$. Since $J_\theta$ is the nested intersection $\bigcap_N f^{\theta_N}(D_N)$ we can find $N$ large enough that $f^{\theta_N}(D_N)$ is contained in $A$. Then $f^{\theta_k}(D_k)$ is a free arc for all $k\geq N$.
\end{proof}

Our next result is a good first step in showing that recurrent points are minimal. It is a modified version of~\cite[Theorem 1.1]{Askri}, and the proof closely follows the one from that paper.

\begin{lem}\label{lem:AskriA}
Let $X$ be a dendrite with $\overline{E(X)}$ countable, $f:X\to X$ a continuous map with zero topological entropy, and $x\in X$ a point which is recurrent but not periodic. Then $\omega(x)$ contains no periodic points.
\end{lem}
\begin{proof}
Throughout the proof we will use freely the following well-known property of $\omega$-limit sets (e.g. \cite{BlockCoppel}): if for fixed $n\geq 2$ we write $W_i=\omega_{f^n}(f^i(x))$ for $0\leq i<n$, then $\omega_f(x)=\bigcup_{i=0}^{n-1} W_i$ and $f(W_i)=W_{i+1 \; (\text{mod }n)}$. In particular, if $\omega_f(x)$ is uncountable, then so is each $W_i$, and if $\omega_f(x)$ contains a given fixed point, then each $W_i$ contains it as well. We continue to use the notation $[x,y]$ for the unique arc in $X$ with endpoints $x,y\in X$, and if $z\in(x,y)=[x,y]\setminus\{x,y\}$ we will say for simplicity that $z$ lies \emph{between} $x$ and $y$.

Now let $L=\omega_f(x)$ where $x$ is recurrent but not periodic. Then $L$ is the closure of the orbit of $x$, hence it is a perfect uncountable set.

\textbf{Step 1:} \emph{$L$ does not contain a periodic point with a free arc neighborhood in $X$}.

For suppose to the contrary that $a\in L$, $f^N(a)=a$, and some free arc $C$ is a neighborhood of $a$ in $X$. Then by the standard properties mentioned above $a\in\omega_{f^N}(f^i(x))$ for some $0\leq i<N$. Replacing $f$ with its iterate and $x$ with its image we may safely assume that $N=1$ and $i=0$, i.e.\ $a$ is a fixed point in $L=\omega_f(x)$.

Since periodic points are never isolated in infinite $\omega$-limit sets we know that $L$ accumulates on $a$ from at least one side in the free arc $C$. So choose an endpoint $b$ of $C$ such that $L\cap [a,b]$ accumulates on $a$. For convenience we let $C$ carry its natural order as an arc, oriented in such a way that $a<b$. Choose five points $y_i\in L\cap [a,b]$ with $a<y_1<y_2<y_3<y_4<y_5<b$. Choose three small arc neighborhoods $I_2, I_3, I_4$ containing $y_2,y_3,y_4$ respectively and let them be pairwise disjoint and lie between $y_1$ and $y_5$. Since the orbit of $x$ visits each of these neighborhoods $I_i$ infinitely often, there must be points in $I_2$ and $I_4$ which visit $I_3$, so by~\cite[Theorem 2.13]{MaiShi} $f$ has a periodic point $c$ between $y_1$ and $y_5$. Replacing $x$ with a point from its orbit near $y_1$ we may assume that $a<x<c<y_5<b$. Let $r$ be the period of $c$ and put $g=f^r$. Since $a$ was already fixed for $f$ we have $a\in\omega_g(x)$ as well. Note that since $x$ is recurrent for $f$ it is also recurrent for $g$.

\textbf{Claim:} \emph{There is an arc $I$ invariant for $g$ with $[a,x]\subseteq I \subseteq [a,c]$}.

To prove the claim put $I=\overline{\bigcup_{n=0}^{\infty} g^n([a,x])}$. Since $a$ is fixed and $x$ is recurrent, it suffices to show that $g^n([a,x]) \subseteq [a,c]$ for all $n$. If this is not true, then there is $z\in[a,x]$ and $n_0 \geq 1$ such that $a$ is between $g^{n_0}(z)$ and $c$ or $c$ is between $g^{n_0}(z)$ and $a$. We treat these two cases separately.

Suppose first that $a$ is between $g^{n_0}(z)$ and $c$. Then $a\in g^{n_0}([z,c])$, so there is $a_{-1}$ between $z$ and $c$ with $g^{n_0}(a_{-1})=a$. Then $f^n(a_{-1})=a$ for all $n\geq n_0\cdot r$. Since $L\cap [a,b]$ accumulates on $a$ we can find a point $x'\in\orb_f(x)$ between $a$ and $a_{-1}$. Since $y_5\in\omega_f(x)$ we can find $n\geq n_0\cdot r$ such that $f^n(x')$ is close to $y_5$ and $a<x'<a_{-1}<c<f^n(x')$. Put $J=[a,x']$ and $K=[x',a_{-1}]$. Then $f^n(J)\cap f^n(K) \supseteq J\cup K$, so $f$ possess an arc horseshoe and thus has positive topological entropy, a contradiction.

Suppose instead that $c$ is between $g^{n_0}(z)$ and $a$. Then $c\in g^{n_0}([a,x])$, so there must be $c_{-1}$ between $a$ and $x$ with $g^{n_0}(c_{-1})=c$. Since $a\in\omega_g(x)$ we can choose $n>n_0$ with $g^n(x)$ close to $a$ so that $g^n(x)<c_{-1}<x<c$. Put $J=[c_{-1},x]$ and $K=[x,c]$. Then again $g^n(J)\cap g^n(K)\supseteq J\cup K$, so $g$ has positive topological entropy and  so does $f$, a contradiction. This completes the proof of the claim.

Now we may use the claim to finish Step 1. Since $x$ belongs to the closed invariant set $I$ we have $\omega_g(x)=\omega_{g|_I}(x)$. But $g|_I$ is an interval map, and when an infinite $\omega$-limit set for an interval map contains a periodic point, the topological entropy must be positive (see~\cite{Mis}), a contradiction.

\textbf{Step 2:} \emph{$L$ does not contain any periodic points.}

Suppose to the contrary that $a\in L$ is a periodic point. As in Step 1 we may assume that $a$ is fixed. By~\cite[Theorem 2.2]{BBPRV} the dendrite $X$ is the union of a countable sequence of free arcs together with a countable set. In particular, we can find a free arc $C$ not containing $a$ with $L\cap C$ uncountable. Write $C=[u,v]$ with $v$ between $u$ and $a$ and let $<$ denote the order in $C$ with $u<v$. Since $L\cap C$ is infinite we may choose four points $x_i\in\orb_f(x)$ with $u<x_1<x_2<x_3<x_4<v$. As in Step 1 we can use small arc neighborhoods of $x_2,x_3,x_4$ to find a periodic point $c$ with $u<x_1<c<v$, and since $x_1$ is in the orbit of $x$ we may redefine $x=x_1$ without changing $\omega_f(x)$. Let $r$ denote the period of $c$ and put $g=f^r$. Since $x$ is recurrent also for $g$ we have $\orb_g(x)\cap[u,c]$ infinite, so we can find two points $x_5,x_6\in\orb_g(x)$ with $u<x_5<x_6<c$ and passing forward along the orbit we can redefine $x=x_6$ without changing $\omega_g(x)$. In particular, $x_5\in\omega_g(x)=\overline{\orb_g(x)}$, so we can choose $p\geq1$ with $g^p(x)$ close to $x_5$ so that $u<g^p(x)<x<c$.

Let $l=\omega_{g^p}(x)$. We have $x\in l$ because $x$ is recurrent and $a\in l$ because $a$ is a fixed point in $\omega_f(x)$. Moreover $c\not\in l$ as a result of Step 1. So let $X_0,$ $X_1$ denote the connected components of $X\setminus\{c\}$ containing $x$ and $a$, respectively, and put $l_i=l\cap X_i$. Then $l=l_0\cup l_1$ expresses $l$ as the disjoint union of two nonempty open subsets (in the topology induced from $X$ to $l$). Recall that every $\omega$-limit set $\omega_f(x)$ is weakly incompressible, i.e. $f(\overline{U})\not\subset U$ for any set $U\subsetneq \omega_f(x)$ open in $\omega_f(x)$ (see, e.g., \cite{SKSF}). Thus we have $g^p(l_0)\cap l_1 \neq \emptyset$. Therefore we may choose $y\in l_0$ with $g^p(y)\in l_1$, and since $\orb_{g^p}(x)$ is dense in $l_0$ we may choose $y$ from the orbit of $x$. We finish the proof in two cases, depending on the location of $y$.

Suppose first that $y$ is between $x$ and $c$. In the ordering of the arc $[g^p(x),g^p(y)]$ we have $g^p(x) < x < y < c < g^p(y)$.  Put $I=[x,y]$ and $J=[y,c]$. Clearly $g^p(I)\supseteq I\cup J$. Since $y\in\orb_g(x)$ we have $\omega_{g}(y)=\omega_{g}(x)\supset\orb_{g}(x)\ni g^p(x)$. In particular, we may choose $n>p$ to make $g^{n}(y)$ as close to $g^p(x)$ as we like, so that $x,y \in [g^{n}(y),c]$. But then $g^{n}(J) \supseteq I\cup J$. We conclude that $g$ possess an arc horseshoe and thus $g$ has positive topological entropy, which is a contradiction with $h_{top}(f)=0$.

Suppose instead that $x$ is between $y$ and $c$. Then $c\in[g^p(x),g^p(y)]$, so there is $c_{-1}\in(x,y)$ with $g^p(c_{-1})=c$. In the ordering of the arc $[y,g^p(y)]$ we have $y<c_{-1}<x<c$ and we also have $x\in(g^p(x),c)$. Put $I=[c_{-1},x]$ and $J=[x,c]$. Since $y\in\orb_g(x)\subset\omega_g(x)$ we can find $n>p$ with $g^n(x)$ as close to $y$ as we like. In particular we can get $x,c_{-1}\in[g^n(x),c]$. But then $g^n(I) \cap g^n(J) \supseteq I\cup J$. Again we conclude that $g$ has positive topological entropy, which is a contradiction. This ends the proof.
\end{proof}

\begin{thm}\label{thm:recmin}
If $X$ is a dendrite in which $\overline{E(X)}$ is countable and if $f\colon X\to X$ has zero topological entropy, then every recurrent point for $f$ is minimal.
\end{thm}

\begin{proof}
Let $x\in\Rec(f)$. If $x$ is periodic then it is minimal, so assume $x$ is not periodic. Let $L=\omega(x)$. Let $M\subset L$ be a minimal set. By Lemma~\ref{lem:AskriA} $L$ contains no periodic orbits, so $M$ is an infinite minimal set. Then Proposition \ref{prop:MinSolSet} applies and we get a sequence of $f$-periodic subdendrites $(D_k)_{k\geq 1}$ and periods $(\alpha_k)$ satisfying properties (1)--(5) of that Proposition. By Lemma~\ref{lm:decomp} for all sufficiently large $k$ we have that $f^i(D_k)$ is a free arc for suitable $i$. Since $M$ is infinite and $D_k$ is periodic, we have $M\cap \Int f^i(D_k)\neq\emptyset$ and as a consequence $Orb_f(x)\cap D_k\neq \emptyset$, for every sufficiently large $k$. Hence $\bigcap_{k\geq 1}Orb_f(D_k)$ contains $L$, that is, property (4) still holds with $L$ in the place of $M$.

We claim that property (5) also holds with $L$ in the place of $M$. Fix $k$ and denote $L_i=f^i(D_k)\cap L$ for $0\leq i<\alpha_k-1$. Observe that $L$ does not contain periodic points, and the set
$\orb(f^i(D_k)\cap f^j(D_k))$ is always finite and invariant for any $i\neq j$ (can be empty)
and hence $L_i \cap f^j(D_k)=\emptyset$ for $i\neq j$. This shows that the sets $L_i \cap L_j=\emptyset$ for $i\neq j$. Clearly $f(L_i) \subseteq L_{i+1 (\text{mod }\alpha_k)}$, and $f(L)=L$ since $\omega$-limit sets are always mapped onto themselves. This shows that $f(L_i)=L_{i+1 (\text{mod } \alpha_k)}$. In particular, we conclude that $L_i$ is uncountable for each $i$.

Again using Lemma~\ref{lm:decomp} choose $k$ large enough that $f^i(D_k)$ is a free arc for some $0\leq i<\alpha_k-1$ and let $A=f^i(D_k)$ denote that free arc. We have just shown that $L_i=A\cap L$ is uncountable, so since $A$ is a free arc there are points from $\omega(x)$ in its interior. Thus we can find a point $y=f^l(x)$ from the forward orbit of $x$ in $A$. Then $\omega_f(x)=\omega_f(y)$ and $y$ is also recurrent for $f$. 
%It is a well known property of recurrent points that 
Since $\Rec(f^{\alpha_k})=\Rec(f)$, $y$ is also recurrent for $f^{\alpha_k}$. But the restriction of $f^{\alpha_k}$ to $A$ is an interval map with zero topological entropy. For such a map, all recurrent points are minimal points, see e.g.~\cite[Chapter VI. Proposition 7]{BlockCoppel}. Thus $y$ is a minimal point for $f^{\alpha_k}$, and hence also for $f$. This shows that $\omega_f(y)=\omega_f(x)$ is a minimal set, and hence $x$ itself is minimal.
\end{proof}

\section{Discrete Spectrum in  Dendrites with $\overline{E(X)}$ Countable}\label{sec:spectrum}

By \cite[Theorem~1.5]{Opr_Quasi} each one-sided subshift with zero entropy can be extended to a dynamical system on the Gehman dendrite
with zero topological entropy. This provides a plethora of examples of dynamical systems on a dendrite with a closed set of endpoints,
having zero topological entropy and invariant measures which do not have discrete spectrum.
But in the Gehman dendrite $E(X)$ is uncountable, since $E(X)$ is a Cantor set. 
On the other hand, each dendrite with $E(X)$ uncountable contains a copy of the Gehman dendrite (e.g. see \cite{Charatonik}, cf. \cite{Opr_Quasi}). So on all these dendrites there exist dynamical systems
with zero topological entropy and invariant measures not having discrete spectrum.

Our work below shows that the opposite holds in the case of a dendrite $X$ where $\overline{E(X)}$ is countable: all invariant measures of zero-entropy mappings have discrete spectrum. So in the case of dendrites, the remaining case in Question~\ref{que:main} is when $E(X)$ is countable but $\overline{E(X)}$ is uncountable. This case is left as a problem for further research.

\begin{lem}\label{lem:simpleds}
Let $(X,f)$ be a topological dynamical system and suppose that all measures $\mu\in M_f(X)$ which are concentrated on $A_i$ have discrete spectrum, for each member $A_i$ of some finite or countable collection of invariant Borel sets. Then any $\mu\in M_f(X)$ which is concentrated on $\bigcup_i A_i$ also has discrete spectrum. In particular, if $\Rec(f)\subseteq \bigcup_i A_i$, then every $\mu\in M_f(X)$ has discrete spectrum.
\end{lem}
\begin{proof}
Let $\mu$ be any finite invariant measure concentrated on $\bigcup_i A_i$. Since each $A_i$ is invariant, i.e. $f(A_i)\subset A_i$, and $f$ preserves $\mu$, we may assume by throwing away a set in $X$ of $\mu$-measure zero that $f^{-1}(A_i)=A_i$ for each $i$.

We may take the index set for the variable $i$ to be $\{1,\ldots,n\}$ in the finite case or $\mathbb{N}$ in the countable case. Then putting $B_i = A_i \setminus \bigcup_{j<i} A_j$ for each $i$, we get a collection $\{B_i\}$ of pairwise disjoint invariant Borel sets. Now let $I=\{i~:~\mu(B_i)>0\}$ and write $\mu_i=\mu|_{B_i}$ for the (unnormalized) restriction of $\mu$ to $B_i$. 
Then we get a direct sum decomposition of Hilbert spaces
$L^2_{\mu}(X) = \bigoplus_{i\in I} L^2_{\mu_i}(B_i).$
We may extend each function $\phi\in L^2_{\mu_i}(B_i)$ to an element of $L^2_{\mu_i}(X)$ by letting $\phi$ vanish outside of $B_i$. Since $f^{-1}(B_i)=B_i$, we see that if $\phi\circ f = \lambda \phi$ holds $\mu_i$ almost-everywhere in $B_i$, then by letting $\phi$ vanish outside $B_i$ it continues to hold $\mu$-almost everywhere in $X$. Thus we have the equivalent direct sum decomposition
\begin{equation}\label{dirsum}
L^2_{\mu}(X) = \bigoplus_{i\in I} L^2_{\mu_i}(X),
\end{equation}
and an eigenfunction in a coordinate space is still an eigenfunction in the whole space. For each $i\in I$, the normalized measure $\mu_i/\mu(B_i)$ is an invariant probability measure for $f$ concentrated on $B_i \subset A_i$, so by hypothesis the eigenfunctions of the Koopman operator on the space $L^2_{\mu_i/\mu(B_i)}(X)$ have dense linear span. Dropping the normalizing constant, the same holds for $L^2_{\mu_i}(X)$. Passing through the direct sum decomposition, it follows that the eigenfunctions of the Koopman operator on the space $L^2_{\mu}(X)$ have dense linear span, that is, $\mu$ has discrete spectrum.

The last statement of the lemma follows by the Poincar\'{e} recurrence theorem, whereby if $\Rec(f) \subseteq \bigcup A_i$, then every measure $\mu\in M_f(X)$ is concentrated on $\bigcup A_i$.
\end{proof}

\begin{lem}
Let $X$ be a dendrite and suppose that $f\colon X\to X$ is a continuous map with zero topological entropy. If $D\subset X$ is a tree and $R\colon X\to D$
is a natural retraction, then the map $F\colon D\to D$
given by $F=R\circ f$ has zero topological entropy.
\end{lem}
\begin{proof}
Suppose that $F$ has positive entropy. Then by~\cite{KKM} there exists an arc horseshoe $I_1,I_2$
with $F^n(I_1\cap I_2) \supset I_1 \cup I_2$ for some $n\in\mathbb N$. Then $F^i(I_j)$ is not a single point for any $i=1,\ldots,n$ and $j=1,2$.
But if $F(J)$ is nondegenerate for an arc $J$ then $f(J)\supset F(J)$ which implies that 
$f^n(I_1)\cap f^n(I_2)\supset I_1\cup I_2$ which implies that $f$ has positive topological entropy. A contradiction.
\end{proof}

\begin{thm}\label{th:RecMinDiscrete}
Let $X$ be a dendrite such that $\overline{E(X)}$ is countable
%{\color{red} Jana: countable and closed (or with countable closure?)} 
and let $f:X\to X$ be a continuous map with zero topological entropy. Then every measure $\mu\in M_f(X)$ has discrete spectrum.
\end{thm}
\begin{proof}
Let $Z=\{z\in \overline{E(X)}~:~\omega_f(z)\text{ is an infinite minimal set}\}.$ Following arguments in \cite[Theorem 10.27]{Nad}, let $(T_n)_{n \in \mathbb{N}} \subset X$ be an increasing sequence of topological trees with endpoints in $E(X)$ defined as follows. We inductively construct the sequence $(T_n)_{n\in \mathbb{N}}$ starting with $T_1 = \{e_1\}$ for some $e_1 \in E(X)$. Then for $n\geq 1$, we attach to $T_n$ an arc $[e,e_{n+1}]$ whose one endpoint $e_{n+1}$ belongs to $E(X)\setminus T_n$ and $e\in T_n$. Since $E(X)$ is countable we can put every endpoint into one of the trees, that is, we let the sequence $(e_n)_{n \in \mathbb{N}}$ be an enumeration of $E(X)$, and then $\bigcup_{n\geq 1} T_n$ being a connected set must coincide with the whole dendrite $X$.

Let $\hat{T}_n=\bigcap_{i=0}^{\infty} f^{-i}(T_n)$ be the maximal invariant set completely contained in $T_n$. Let $\Per(f)$ be the set of periodic points of $f$. We claim that:
\begin{equation}\label{cover}
    \Rec(f) \subset \Per(f) \cup \left(\bigcup_{z\in Z}\omega_f(z)\right) \cup \left(\bigcup_{n} \hat{T}_n\right).
\end{equation}
To see this, let $x$ be a non-periodic recurrent point whose orbit is not contained in any of the trees $T_n$. This means that there are points $f^{n_i}(x)$ which belong to $T_{m_i}\setminus T_{m_i-1}$ for some strictly increasing sequences $m_i$, $n_i \to \infty$. Then the arcs $[f^{n_i}(x), e_{m_i}]$ in $X$ are pairwise disjoint, so by~\cite[Lemma 2.3]{MaiShi} their diameters tend to zero. This shows that $\liminf_{n\to\infty} d(f^n(x),E(X))=0$. Therefore $\omega_f(x)\cap \overline{E(X)}\neq\emptyset$.
By Theorem~\ref{thm:recmin}, $\omega_f(x)$ is a minimal set, so choosing $z\in\omega_f(x)\cap \overline{E(X)}$ we 
have $\omega_f(x)=\omega_f(z)$. This establishes~\eqref{cover}.

Now observe that any finite invariant measure concentrated on $\Per(f)$ has discrete spectrum, see eg.~\cite[Theorem 2.3]{Opr_Quasi}. As for the sets $\hat{T}_n$,  note that for each $n\in \mathbb{N}$ the map $F=R\circ f$, where $R\colon X\to T_n$ is a retraction, satisfies $F|_{\hat{T}_n}=f|_{\hat{T}_n}$ by the definition 
and therefore each $f$-invariant measure concentrated on $\hat{T}_n$ (a subset of a tree) has discrete spectrum, as, by \cite{Opr_Quasi}, all invariant measures of $F$ have discrete spectrum.

Finally, we claim that any invariant measure concentrated on $\omega_f(z)$, $z\in Z$, has discrete spectrum. Let $(D_k)$ be the periodic subdendrites with periods $(\alpha_k)$ described in Proposition~\ref{prop:MinSolSet} and let $\pi:(\omega_f(z),f)\to(\Omega,\tau)$ be the factor map onto the odometer described in Lemma~\ref{lem:almostinvertible}. Let $\mu\in M_f(X)$ be any invariant measure concentrated on $\omega_f(z)$. Then the pushforward measure $\pi_*(\mu)$ is invariant for the odometer, so by unique ergodicity it is the Haar measure on $\Omega$ and by well-known properties of odometers it has discrete spectrum. Now since $\omega(z)$ contains no periodic points, we know that $\mu$ is non-atomic and therefore countable sets have measure zero. Then in the category of measure preserving transformations, the factor map $\pi:(\omega_f(z),\mu,f)\to(\Omega,\pi_*(\mu),\tau)$ is in fact an isomorphism, since by Lemma~\ref{lem:almostinvertible} it is invertible except on a set of $\mu$-measure zero. This implies that $\mu$ has discrete spectrum.

We have shown that an invariant measure concentrated on any of the countably many invariant sets in~\eqref{cover} has discrete spectrum. By Lemma~\ref{lem:simpleds} this completes the proof.
\end{proof}

\section*{Acknowledgements}

M. Foryś-Krawiec was supported in part by the National Science Centre, Poland (NCN), grant SONATA BIS no. 2019/34/E/ST1/00237: "Topological and Dynamical Properties in Parameterized Families of Non-Hyperbolic Attractors: the inverse limit approach".

S. Roth was supported by Czech Republic RVO funding for I\v{C}47813059.
\begin{table}[h]

\begin{tabular}[t]{b{1.5cm} m{13.5cm}}

\includegraphics [width=.09\textwidth]{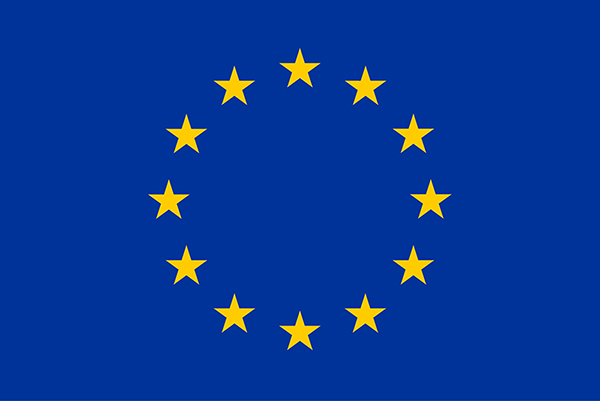} & 
This research is part of a project that has received funding from the European Union's Horizon 2020 research and innovation programme under the Marie Sk\l odowska-Curie grant agreement No 883748.

\end{tabular}
\end{table}

\end{document}